\documentclass[12pt,oneside]{article}
\usepackage[centertags]{amsmath}
\usepackage{amsfonts}
\usepackage{amssymb}
\usepackage{amsthm}
\usepackage{mathrsfs}
\usepackage{setspace}
\usepackage[papersize={8.5in,11in}]{geometry}
\def\N{\mathbb{N}}

\newtheorem{thm}{Theorem}[section]
\newtheorem{cor}[thm]{Corollary}
\newtheorem{lem}[thm]{Lemma}

\numberwithin{equation}{section}

\numberwithin{equation}{section}

\title{Completely monotonic ratios of basic and ordinary gamma functions}
\author{Christian\:Berg$^{\rm a}$, Asena\:\c{C}etinkaya$^{\rm b}$, Dmitrii\:Karp$^{\rm c}$\footnote{Corresponding author. E-mail: D. Karp -- \emph{dimkrp@gmail.com}, Ch.\:Berg -- \emph{berg@math.ku.dk}, A.\:\c{C}etinkaya --  \emph{asnfigen@hotmail.com}}~
\\[10pt]\small{\textit{$\phantom{1}^a$University of Copenhagen, Denmark}}\\\small{\textit{$\phantom{1}^b$\.{I}stanbul K\"{u}lt\"{u}r University, Turkey}}
\\\small{\textit{$\phantom{1}^c$Holon Institute of Technology, Israel}}}
\date{}

\begin{document}

\maketitle

\begin{center}
\parbox{12cm}{
\small\textbf{Abstract.}
We investigate conditions for logarithmic complete monotonicity of product ratios of gamma and $q$-gamma
functions whose arguments are linear functions of the variable.  We give necessary and sufficient conditions in terms of nonnegativity of a certain explicitly written measure in the $q$ case and of a certain elementary function in the classical $q=1$ case. In the latter case we further provide simple new sufficient conditions leading to many new examples of logarithmically completely monotonic gamma ratios. Finally, we apply some of our results to study monotonicity of some gamma ratios and rational functions.}
\end{center}

\bigskip


\bigskip


\bigskip

\section{Introduction}
Completely monotonic functions are infinitely differentiable non-negative functions defined on $(0,\infty)$ such that
$(-1)^nf^{(n)}(x)\ge0$ for $n\ge1$ and $x>0$ \cite[Definition~1.3]{SSV}. They are characterized in Bernstein's theorem
as Laplace's transforms of nonnegative measures \cite[Theorem~1.4]{SSV}. These functions are of importance in many fields including probability, potential theory, asymptotic analysis and combinatorics. Details and references can be found, for instance, in a nice survey \cite{Koumandos}.  A positive function $f$ is said to be logarithmically completely
monotonic (l.c.m.) if $-(\log{f})'$ is completely monotonic \cite[Definition~5.8]{SSV}.  The class of l.c.m. functions is a proper subset of the class of c.m. functions. Their importance stems from the fact that they represent Laplace transforms of infinitely divisible distributions, see \cite[Theorem~5.9]{SSV}.

The study of (logarithmic) complete monotonicity for ratios of products of gamma functions was probably initiated by Bustoz and Ismail in their 1986 paper \cite{BusIsm}. Today this topic has a rather rich literature
that includes contributions made by Ismail and Muldoon \cite{IsmailMuldoon94}, Alzer \cite{Alzer}, Grinshpan and Ismail \cite{GrinIsm}, Alzer and Berg \cite{AlzBerg}, Guo and Qi \cite{Guo-Qi,QiGuo2020}, Karp and Prilepkina \cite{KPCMFT2016}, Ouimet \cite{Ouimet2018,Ouimet2019} and Qi et.al.\cite{QiLiYuDuGuo2019}. The original result of Bustoz and Ismail has been recently substantially strengthened by Berg, Koumandos and Pedersen in \cite{BergKoumPedrsen}.

More specifically, the third named author (jointly with E.\:Prilepkina) considered in \cite{KPCMFT2016} the function
\begin{equation}\label{eq:KP2016}
x\to\frac{\prod_{i=1}^p\Gamma(A_ix+a_i)}{\prod_{j=1}^s\Gamma(B_jx+b_j)},
\end{equation}
where $A=(A_1,\ldots,A_p)$ and $B=(B_1,\ldots,B_s)$ are strictly positive scaling factors, while $a=(a_1,\ldots, a_p)$ and $b=(b_1,\ldots,b_s)$ are non-negative.  Conditions for this function to be logarithmically completely monotonic and the representing measure were found.  The first purpose of this paper is to extend some of the results of \cite{KPCMFT2016} to the ratio
\begin{equation}\label{eq:Wq}
W_q(x)=\frac{\prod_{i=1}^p\Gamma^{\alpha_i}_q(A_ix+a_i)}{\prod_{j=1}^s\Gamma^{\beta_j}_q(B_jx+b_j)},
\end{equation}
where the $q$-gamma function $\Gamma_q(x)$ is defined by \cite[(1.1)]{Askey}, \cite[(1.10.1)]{GR}
\begin{equation}\label{eq:Gammaq-defined}
\Gamma_q(x)=(1-q)^{1-x}\prod_{n=0}^\infty\frac{1-q^{n+1}}{1-q^{x+n}}.
\end{equation}
Here $q\in(0,1)$ and  $A_i$, $\alpha_i$, $\beta_j$ and $B_j$ are strictly positive, while $a_i$ and $b_j$ are non-negative.
Investigations of the complete monotonicity of the ratios of $q$-gamma functions are not as numerous as of their classical (i.e. $q=1$) counterparts. Probably, the first appearance is encountered in the 1986 paper by Ismail, Lorch and  Muldoon \cite[Theorem~2.5, Theorem~6.1(vi)]{IsmLorchMuld} with further results in \cite[Theorem~2.5]{IsmailMuldoon94} (we refer to the corrected 2013 version), extending a previous result by Bustoz and Ismail pertaining to $q=1$ case. Further combinations involving the ratio of two $q$-gamma functions were studied by the same authors in \cite{IsmailMuldoon2013}.  Grinshpan and Ismail considered $p=s$, $\alpha_i=\beta_j=A_i=B_j=1$ case of \eqref{eq:Wq} in \cite{GrinIsm}, where some sufficient conditions for the logarithmic completely monotonicity were given (see further comments regarding this paper in Example~1 in Section 2 below). Some completely monotonic combinations of $q$-gamma and $q$-digamma functions were also considered in \cite{Gao}.

In this paper we give necessary and sufficient conditions for the function $W_q$ to be logarithmically completely monotonic and furnish several examples of how these conditions can be satisfied.  The $q$ case turns out to be significantly different from the classical $q=1$ case.  For example, the balancing condition $\sum_{i=1}^{p}\alpha_{i}A_{i}=\sum_{j=1}^{s}\beta_{j}B_{j}$ necessary in the classical situation may be violated when $q\in(0,1)$ without destroying the logarithmic complete monotonicity of $W_q$; further for $q\in(0,1)$ the arithmetic properties of the numbers $A_i$, $B_j$ play the key role in determining whether $W_q$ can be logarithmically completely monotonic.

In the second part of the paper we consider the $q=1$ case of \eqref{eq:Wq} which extends \eqref{eq:KP2016} (since $\Gamma_q(x)\to\Gamma(x)$ as $q\uparrow1$, see \cite[Theorem~4.1]{Askey}). We show first that the results of \cite{KPCMFT2016} can be extended \emph{mutatis mutandis} to this situation. In general, the necessary and sufficient conditions for logarithmic complete monotonicity are difficult to verify.  We give several new examples which are completely monotonic under very simple and explicit conditions. We further demonstrate how these examples can be used to prove monotonicity of certain gamma ratios and rational functions.

We conclude this introduction by recalling some definitions to be used below.
The standard notation $\Gamma(x)$ and $\psi (x)$ will be used for Euler's gamma function and it logarithmic derivative (known as digamma or psi function), respectively.   The following integral representation holds \cite[(1.1)]{AlzBerg}
\begin{equation}\label{eq:psi(x)}
\psi(x)=-\gamma+\int_0^\infty \frac{e^{-t}-e^{-xt}}{1-e^{-t}}dt,
\end{equation}
where $\gamma=0.57721\ldots$ denotes the Euler-Mascheroni constant. Polygamma functions are
\begin{equation}\label{eq:psi^k(x)}
\psi^{(k)}(x)=(-1)^{k+1}\int_0^\infty \frac{t^ke^{-xt}}{1-e^{-t}}dt.
\end{equation}

The $q$-gamma function is defined in \eqref{eq:Gammaq-defined} and its logarithmic derivative (or $q$-digamma function) $\psi_q(x)$ can be computed for $q\in(0,1)$, $x>0$, by the formulas \cite[(3.3)]{Askey},  \cite[(1.3),(1.4)]{IsmailMuldoon2013}
\begin{align}\label{eq:psiq-defined}
\psi_q(x)=\frac{\Gamma'_q(x)}{\Gamma_q(x)}&=-\log(1-q)+\log(q)\sum_{n=0}^\infty\frac{q^{n+x}}{1-q^{n+x}}
\nonumber\\
&=-\log(1-q)+\log(q)\sum_{n=1}^\infty\frac{q^{nx}}{1-q^{n}}.
\end{align}
The second expression can be rewritten as a Stieltjes integral \cite[(1.5)]{IsmailMuldoon2013}
\begin{equation}\label{eq5}
\psi_q(x)=-\log(1-q)-\int_0^\infty\frac{e^{-xt}}{1-e^{-t}}d\gamma_q(t),
\end{equation}
where
$$
\gamma_q(t)= \left\{
\begin{array}{cc}
\log(1/q)\sum\limits_{k=1}^{\infty} \delta_{k\log(1/q)}, & 0<q<1
\\[10pt]
t, & q=1.
\end{array}
\right. \\
$$
and $\delta_{x}$ denotes the unit mass at the point $x$ so that $d\gamma_q(t)$ is a discrete measure with positive masses $\log(1/q)$ at the positive points $k\log(1/q)$,  $k\ge1$.

Differentiating both sides of (\ref{eq:psiq-defined}) and (\ref{eq5}), respectively,  yields
\begin{equation}\label{eq:psiDk}
\psi_q^{(k)}(x)=[\log(q)]^{k+1}\sum_{n=1}^\infty\frac{n^kq^{nx}}{1-q^{n}}
=(-1)^{k+1}\int_0^\infty\frac{t^ke^{-xt}}{1-e^{-t}}d\gamma_q(t),\quad\quad q\in(0,1), x>0.
\end{equation}
Further details can be found in \cite{IsmailMuldoon94,IsmailMuldoon2013}.

Finally, we remark that the definition of a completely monotonic function leads immediately to the following equivalences
\begin{align}\label{eq:c.m}
f\ \text{is c.m. on } (0,\infty) &\Leftrightarrow f\geq0\ \text{and} -f'\  \text{is c.m. on } (0,\infty)\nonumber\\
&\Leftrightarrow -f'\  \text{is c.m. on } (0,\infty)\ \text{and} \lim_{x\rightarrow\infty}f(x)\geq0.
\end{align}
In view of (\ref{eq:c.m}) for logarithmically completely monotonic functions we have:
\begin{align}\label{eq:l.c.m}
f\ \text{is l.c.m. on } (0,\infty) &\Leftrightarrow (-\log f(x))'\geq0\ \text{and}\ (\log f)''\  \text{is c.m. on } (0,\infty)\nonumber\\
&\Leftrightarrow (\log f)''\  \text{is c.m. on } (0,\infty)\ \text{and}\ \lim_{x\rightarrow\infty}(-\log f(x))'\geq0.
\end{align}

\section{Ratios of $q$-gamma functions}
In this section we consider the function $x\to W_q(x)$ defined in \eqref{eq:Wq} with $q\in(0,1)$ and  $A_i$, $\alpha_i$, $\beta_j$ and $B_j$ being strictly positive scaling factors, while $a_i$ and $b_j$ are assumed to be non-negative.  Given the numbers $A_i$ and $B_j$ and $0<q<1$, define two multi-sets of positive numbers by
$$
\mathcal{A}=\left\{nA_i\log(1/q): n\in\N, i=1,\ldots,p\right\}
$$
$$
\mathcal{B}=\left\{mB_j\log(1/q): m\in\N, j=1,\ldots,s\right\}.
$$
We will write $\hat{\mathcal{A}}$, $\hat{\mathcal{B}}$ for the sets obtained by removing the repeated elements in the multi-sets $\mathcal{A}$, $\mathcal{B}$, respectively. Define further the (non-negative) measure $\mu$ supported on $\hat{\mathcal{A}}$ by
\begin{equation}\label{eq:mu-defined}
\mu=\sum_{i=1}^{p}\sum_{n=1}^\infty \frac{n}{1-q^n}\alpha_iA_i^2q^{na_i}\delta_{nA_i\log(1/q)},
\end{equation}
where $\delta_{a}$ denotes the point mass concentrated at the point $a$.  Similarly, define
the (non-negative) measure $\sigma$ supported on $\hat{\mathcal{B}}$ by
\begin{equation}\label{eq:sigma-defined}
\sigma=\sum_{j=1}^{s}\sum_{m=1}^\infty \frac{m}{1-q^m}\beta_jB_j^2q^{mb_j}\delta_{mB_j\log(1/q)}.
\end{equation}

\begin{thm}\label{th:2derivative}
Suppose $A_i$, $B_j$, $\alpha_i$ and $\beta_j$ are strictly positive, while $a_i$ and $b_j$ are non-negative. The function $(\log W_q)''$ is completely monotonic if and only if
$\tau:=\mu-\sigma$ is a non-negative measure, which, in turn, is equivalent to the following two conditions:

\emph{(1)} $\hat{\mathcal{B}}\subset\hat{\mathcal{A}}~\Leftrightarrow~\exists\: n_1,\ldots,n_s\in\N$ such that
$B_1=n_1A_{i_1}$, $B_2=n_2A_{i_2}$, $\ldots$, $B_s=n_sA_{i_s}$, for some $i_1,\ldots, i_s\in \{1,2,\ldots,p\}$ \emph{(}the indices $i_r,i_l$ may coincide\emph{)}.

\emph{(2)} for each $t\in\hat{\mathcal{B}}$ we have
\begin{equation}\label{eq:masscondition}
\sum\limits_{\substack{{n,i}\\{nA_i\log(1/q)=t}}}\frac{n}{1-q^n}\alpha_iA_i^2q^{na_i}
\ge\sum\limits_{\substack{{m,j}\\{mB_j\log(1/q)=t}}}\frac{m}{1-q^m}\beta_jB_j^2q^{mb_j}.
\end{equation}

In the affirmative case we have
\begin{equation}\label{eq:logW2prime}
(\log W_q)''=(\log q)^2\int_0^\infty e^{-xy} \tau(dy).
\end{equation}
\end{thm}
\begin{proof} Differentiating the logarithm of \eqref{eq:Wq} twice and using \eqref{eq:psiDk}, we have
\begin{align}
	(\log W_q(x))''&=\sum_{i=1}^p\alpha_iA^2_i\psi'_q(A_ix+a_i)-\sum_{j=1}^s\beta_jB^2_j\psi'_q(B_jx+b_j)\nonumber\\
	&=\sum_{i=1}^p\alpha_iA^2_i\log^2 q\sum_{n=1}^\infty\frac{nq^{n(A_ix+a_i)}}{1-q^n}-\sum_{j=1}^s\beta_jB^2_j\log^2 q\sum_{n=1}^\infty\frac{nq^{n(B_jx+b_j)}}{1-q^n}.
\end{align}
This an be rewritten as
$$
(\log W_q(x))''=(\log q)^2\int_0^\infty e^{-xy} \tau(dy),
$$
where $\tau=\mu-\sigma$ with $\mu$ and $\sigma$ given by \eqref{eq:mu-defined} and \eqref{eq:sigma-defined}, respectively.
By Bernstein's theorem $(\log W_q(x))''$ is completely monotonic if and only if $\tau$ is a non-negative measure, which proves the first claim. To show that this is equivalent to the combination of conditions (1) and (2) note that nonnegativity of $\tau$ reduces to the requirement that each of the negative mass points $mB_j\log(1/q)$ is among the positive mass points $nA_i\log(1/q)$ and the mass at the latter point is greater or equal the mass at the former.  The first requirement amounts to $\hat{\mathcal{B}}\subset\hat{\mathcal{A}}$ in view of the definitions of $\hat{\mathcal{A}}$ and $\hat{\mathcal{B}}$, while the second one is precisely the inequality \eqref{eq:masscondition}.

Finally, for $\hat{\mathcal{B}}\subset\hat{\mathcal{A}}$ to hold it is necessary and sufficient that there exist $n_1,\ldots,n_s\in\N$ such that
\begin{equation}\label{eq:BviaA}
B_1=n_1A_{i_1}, B_2=n_2A_{i_2}, \ldots, B_s=n_sA_{i_s},
\end{equation}
for some $i_1,\ldots, i_s\in \{1,2,\ldots p\}$ (the indices $i_r,i_l$ may coincide). Indeed, we can take the part of $\hat{\mathcal{B}}$ corresponding to  $m=1$ to see that this condition is necessary. On the other hand, if this condition is satisfied it is clearly also sufficient as $mB_{k}=nA_{i_k}$ for $n=mn_k$.
\end{proof}

Conditions \eqref{eq:masscondition} may be simplified in the following case.  Suppose all ratios $B_i/B_j$, $i\ne{j}$, are irrational, so that $\mathcal{B}=\hat{\mathcal{B}}$ ($\mathcal{B}$ has no repeated elements) and suppose that
$A_i=B_i/n_i$ for $i=1,\ldots,s$, while all elements of $(A_{s+1},\ldots,A_p)$ are irrational with respect to any element of $(A_{1},\ldots,A_{s})$ (that is $A_k/A_l$ is irrational if $l\le{s}$, $k>s$). This implies that each equation $nA_i=mB_j$ is satisfied by a unique combination $(n,i)$, $(m,j)$. Then condition \eqref{eq:masscondition} reduces to a set of inequalities
$$
\frac{\alpha_{j}A_{j}q^{na_{j}}}{1-q^n}\ge\frac{\beta_jB_jq^{mb_j}}{1-q^m},\quad j=1,\ldots,s,\; m\in\N.
$$
Inserting $B_j=n_jA_{j}$, $n=mn_j$ here,  these conditions reduce to
$$
\frac{\alpha_{j}q^{mn_ja_{j}}}{1-q^{mn_j}}\ge\frac{\beta_jn_jq^{mb_j}}{1-q^m},\quad j=1,\ldots,s,\; m\in\N.
$$
Writing $\epsilon=q^m\in(0,q]$ we have
$$
\frac{\alpha_{j}\epsilon^{n_ja_{j}}}{1-\epsilon^{n_j}}\ge\frac{\beta_jn_j\epsilon^{b_j}}{1-\epsilon},\quad j=1,\ldots,s.
$$
or
$$
\frac{\alpha_{j}}{\beta_jn_j}\ge(1+\epsilon+\cdots+\epsilon^{n_j-1})\epsilon^{b_j-n_ja_{j}},\quad j=1,\ldots,s.
$$
Clearly, we need to assume $b_{j}\ge{n_ja_{j}}$ in order to satisfy this inequality for $\epsilon$ close to $0$ (that is for large $m$).  Further, under this assumption the maximum of the right hand side is attained for $m=1$ (or $\epsilon=q$), so it is sufficient to choose the parameters satisfying
\begin{equation}\label{eq:ABprime-condition}
\frac{\alpha_{j}}{\beta_j}\ge n_jq^{b_j-n_ja_{j}}\frac{1-q^{n_j}}{1-q},\quad j=1,\ldots,s.
\end{equation}
For any given values of $q\in(0,1)$ and $n_j$, we can choose $\beta_j$, $\alpha_{j}$, $b_j$ and $a_{j}$ satisfying these inequalities.

If $(\log W_q)''$ is completely monotonic on $(0,\infty)$ two types of behavior of $(\log W_q)'$ seem to be of interest: if $(\log W_q)'\ge0$ then $(\log W_q)'$ is known as a Bernstein function \cite{SSV}; if $(\log W_q)'\le0$ then $(-\log W_q)'$ is completely monotonic and, hence $W_q$ is logarithmically completely monotonic.
We will explore these two cases (with the main emphasis on the latter case) in the following two theorems.  For additional clarity recall that
$$
(\log W_q)'=\sum\limits_{i=1}^{p}\alpha_{i}A_{i}\psi_q(A_{i}x+a_i)-\sum\limits_{j=1}^{s}\beta_{j}B_{j}\psi_q(B_{j}x+b_j).
$$

\begin{thm}\label{th:WqBernstein}
The function $(\log W_q)'$ is a Bernstein function if and only if conditions \emph{(1)} and \emph{(2)} from Theorem~\emph{\ref{th:2derivative}} hold and
$$
\sum\limits_{i=1}^{p}\alpha_{i}A_{i}\psi_q(a_i)\ge\sum\limits_{j=1}^{s}\beta_{j}B_{j}\psi_q(b_j)
$$
\end{thm}
\noindent\textbf{Proof.}  As $(\log W_q)''\ge0$ by Theorem~\ref{th:2derivative}, the function $(\log W_q)'$ is increasing. Hence, $(\log W_q)'\ge0$ if and only if $\lim\limits_{x\to0}(\log W_q(x))'\ge0$. This is exactly the condition of the theorem. $\hfill\square$

\begin{thm}\label{th:WqLCM}
The function $W_q(x)$ is logarithmically completely monotonic if and only if conditions \emph{(1)} and \emph{(2)} from Theorem~\emph{\ref{th:2derivative}} hold and
\begin{equation}\label{eq:LCMcondition}
\sum\limits_{i=1}^{p}\alpha_{i}A_{i}\le\sum\limits_{j=1}^{s}\beta_{j}B_{j}.
\end{equation}
\end{thm}
\noindent\textbf{Proof.}  Suppose  $W_q(x)$ is logarithmically completely monotonic, so
that $(-\log W_q)'$ is completely monotonic. This  implies that $(\log W_q)''$  is completely monotonic and by Theorem~\ref{th:2derivative} conditions (1) and (2) of this theorem must be satisfied.  Further, as $(-\log W_q)'$ must be non-negative and decreasing, we necessarily get with the help of (\ref{eq:psiq-defined})
$$
\lim\limits_{x\to\infty}(-\log W_q(x))'=
\left(\sum\limits_{j=1}^{s}\beta_{j}B_{j}-\sum\limits_{i=1}^{p}\alpha_{i}A_{i}\right)\log\left(\frac{1}{1-q}\right)\ge0
$$
which is \eqref{eq:LCMcondition}.  In opposite direction conditions (1) and (2) imply that  $(\log W_q)''$  is completely monotonic by Theorem~\ref{th:2derivative}, so that  $(-\log W_q)'$ is decreasing. Then condition \eqref{eq:LCMcondition} implies in view of the above limit that $(-\log W_q)'\ge0$ and hence is completely monotonic. $\hfill\square$

The above theorem shows that the situation in the $q$-case is substantially different from the $q=1$ case, where condition \eqref{eq:LCMcondition} must be satisfied with equality sign.
Below, we present three examples.

\textbf{Example~1.} In \cite{GrinIsm}  Grinshpan and Ismail considered the ratio \eqref{eq:Wq} with $A_i=B_j=1$:
$$
F_q(x)=\frac{\prod_{i=1}^p\Gamma^{\alpha_i}_q(x+a_i)}{\prod_{j=1}^s\Gamma^{\beta_j}_q(x+b_j)}
$$
(we follow our notation which is slightly different from the notation of \cite{GrinIsm}).  In this case the measure $\mu$ from \eqref{eq:mu-defined} assigns the mass
$$
\sum_{i=1}^{p} \frac{n}{1-q^n}\alpha_iq^{na_i}
$$
to the point $n\log(1/q)$ for each $n\in\N$, while the measure $\sigma$ from \eqref{eq:sigma-defined} assigns the mass
$$
\sum_{j=1}^{s}\frac{n}{1-q^n}\beta_jq^{nb_j}
$$
to the same point, so that condition (1) from Theorem~\ref{th:2derivative} is automatically satisfied, while condition (2) reduces to
$$
\sum\limits_{i=1}^{p}\alpha_iq^{na_i}-\sum\limits_{j=1}^{s}\beta_jq^{nb_j}=v(q^n)\ge0
$$
for each $n\in\N$, where
$$
v(t)=\sum\limits_{i=1}^{p}\alpha_it^{a_i}-\sum\limits_{j=1}^{s}\beta_jt^{b_j}.
$$
In particular, $v(t)\ge0$ on $(0,q)$ is sufficient.  The condition \eqref{eq:LCMcondition} takes the form
$$
\sum\limits_{i=1}^{p}\alpha_{i}\le\sum\limits_{j=1}^{s}\beta_{j}.
$$
Hence, the above inequality and the condition $v(q^n)\ge0$ for all $n\in\N$ are necessary and sufficient for $F_q$ to be logarithmically completely monotonic.  This refines \cite[Lemma~3.1]{GrinIsm}, where the claim is that $F_q$ is logarithmically completely monotonic if $\sum_{i=1}^{p}\alpha_{i}=\sum_{j=1}^{s}\beta_{j}$ and $v(t)\ge0$ on $(0,1)$.

Note that the above inequality may be strict without contradicting  $v(t)\ge0$ on $(0,q)$. For instance,
$$
x\to\frac{\Gamma^{\alpha}_q(x+a)}{\Gamma^{\beta}_q(x+b)}
$$
is logarithmically completely monotonic for $\beta>\alpha$ if $\alpha/\beta\ge{q^{b-a}}$ which can be attained for any $q\in(0,1)$ by choosing $b$ sufficiently large.

\textbf{Example~2.} Take $A_1=1/6$, $\alpha_1=5$, $B_1=1/3$, $B_2=1/2$, $\beta_1=\beta_2=1$, or
$$
W_q(x)=\frac{\Gamma_{q}^5(x/6+a)}{\Gamma_{q}(x/3+b)\Gamma_{q}(x/2+c)}.
$$
Note that $\alpha_1A_1=\beta_1B_1+\beta_2B_2$. We will try to choose $a$, $b$, $c$ in order that $(\log{W_q})''$ is completely monotonic, that is we try to satisfy \eqref{eq:masscondition}.
Denote $\delta=\log(1/q)>0$. Then we have
$$
\mathcal{A}=\left\{\delta{n}/6: n\in\N\right\},~~~
\mathcal{B}=\left\{\delta{m_1}/3, \delta{m_2}/2: m_1,m_2\in\N\right\}.
$$
All elements of $\mathcal{A}$ are simple (not repeated), but the elements $t$ of $\mathcal{B}$ of the form $t=\delta{j}$, $j\in\N$, have multiplicity $2$, as both values $m_1=3j$ and $m_2=2j$, $j\in\N$, give $t=\delta{j}$.  Hence, for $t=\delta{m_1}/3=\delta{n}/6$, $m_1\ne0\!\!\mod\!3$, we have $n=2m_1$ and condition \eqref{eq:masscondition} reduces to
$$
\frac{2m_1}{1-q^{2m_1}}5(1/6)^2q^{2m_1a}\ge\frac{m_1}{1-q^{m_1}}(1/3)^2q^{m_1b}~\Leftrightarrow~
5/2\ge(1+q^{m_1})q^{m_1(b-2a)}
$$
We have to assume $b-2a\ge0$ to satisfy this inequality for large $m_1$.  Assuming this, we see that the right hand side is decreasing as $m_1$ is increasing. Hence, it is sufficient to satisfy this inequality for $m_1=1$ or
$$
5/2\ge(1+q)q^{b-2a}~\Leftrightarrow~b-2a\ge\frac{\log[2(1+q)/5]}{\log(1/q)}.
$$
The last inequality is true for all  $b-2a\ge0$.

Next for $t=\delta{m_2}/2=\delta{n}/6$, $m_2\ne0\!\!\mod\!2$, we have $n=3m_2$ and condition \eqref{eq:masscondition} reduces to
$$
\frac{3m_2}{1-q^{3m_2}}5(1/6)^2q^{3m_2a}\ge\frac{m_2}{1-q^{m_2}}(1/2)^2q^{m_2c}~\Leftrightarrow~
5/3\ge(1+q^{m_2}+q^{2m_2})q^{m_2(c-3a)}.
$$
We have to assume $c-3a\ge0$ to satisfy this inequality for large $m_2$.  Assuming this, we see that the right hand side is decreasing as $m_2$ is increasing. Hence, it is sufficient to satisfy this inequality for $m_2=1$ or
$$
5/3\ge(1+q+q^2)q^{c-3a}~\Leftrightarrow~c-3a\ge\frac{\log[3(1+q+q^2)/5]}{\log(1/q)}.
$$
We can always take $c$ large enough to satisfy this inequality.

Finally take the points $t\in\mathcal{B}$ of the form $t=\delta{j}$, so that $\delta{m_1}/3=\delta{n}/6$ and $\delta{m_2}/2=\delta{n}/6$ (i.e. $m_1=3,6,9,\ldots$ and $m_2=2,4,6,\ldots$).  Condition \eqref{eq:masscondition} takes the form:
$$
\frac{n}{1-q^n}5(1/6)^2q^{na}\ge\frac{m_1}{1-q^{m_1}}(1/3)^2q^{m_1b}+\frac{m_2}{1-q^{m_2}}(1/2)^2q^{m_2c}.
$$
As $m_1=n/2$, $m_2=n/3$ this can be put into the form:
$$
5\ge2(1+q^{n/2})q^{n(b-2a)/2}+3(1+q^{n/3}+q^{2n/3})q^{n(c-3a)/3}.
$$
We have to assume again $b-2a\ge0$ and $c-3a\ge0$. Under this assumption the right hand side decreases as $n$ grows.  So it suffices to take the minimum value of $n=6$ (this is minimal value producing integer $n/6$).  So, the required inequality takes the form
$$
5\ge2(1+q^{3})q^{3(b-2a)}+3(1+q^{2}+q^{4})q^{2(c-3a)}
$$
It is clear that for any given $0<q<1$ we can choose $b-2a$ and $c-3a$ large enough to satisfy this inequality.  We finally conclude that for $a$, $b$, $c$ satisfying the above inequality and such that
$$
b\ge{2a}, ~~~~c\ge 3a+\frac{\log[3(1+q+q^2)/5]}{\log(1/q)},
$$
the function $W_q(x)$ is l.c.m.

\bigskip

\textbf{Example~3}. Take  $A_1=\sqrt{2}/2$, $A_2=\sqrt{3}/5$, $A_3=\pi$, $A_4=1$, $B_1=\sqrt{2}$, $B_2=\sqrt{3}$, so that
$B_1=2A_1$, $B_2=5A_3$, $\mathcal{B}$ has no repeated elements and, clearly, the elements of $(A_3,A_4)$ are irrational with respect to the elements of $(A_1,A_2)$.  Then according to  \eqref{eq:ABprime-condition} we need the parameters to satisfy
$$
\alpha_{1}/\beta_{1}\ge 2(1+q)q^{b_{1}-2a_{1}}~~~\text{and}~~~\alpha_{2}/\beta_{2}\ge 5q^{b_{2}-5a_{2}}(1-q^{5})/(1-q)
$$
in order that $(\log W_{q})''$ be completely monotonic.  Finally, according to condition \eqref{eq:LCMcondition} of Theorem~\ref{th:WqLCM} the function $W_{q}$ is logarithmically completely monotonic if, in addition to the above conditions, we have
$$
\alpha_{1}\sqrt{2}/2+\alpha_{2}\sqrt{3}/5+\alpha_{3}\pi+\alpha_{4}\le \beta_{1}\sqrt{2}+\beta_2\sqrt{3}.
$$
For instance,  for $q=1/2$, the function
$$
x\to\frac{\Gamma_{q}(x/\sqrt{2}+a_1)\Gamma_{q}(x\sqrt{3}/5+a_2)\Gamma_{q}(x\pi+a_3)\Gamma_{q}(x+a_4)}
{\Gamma_{q}^2(x\sqrt{2}+2a_{1}+3)\Gamma_{q}^2(x\sqrt{3}+5a_{2}+5)}
$$
is logarithmically completely monotonic for any $a_1,a_2,a_3,a_4\ge0$.

\section{The $q=1$ case revisited}

In this section we will study the function
\begin{equation}\label{eq:V-defined}
V(x)=\theta^{-x}\frac{\prod_{i=1}^p\Gamma^{\alpha_{i}}(A_{i}x+a_i)}{\prod_{j=1}^s\Gamma^{\beta_{j}}(B_{j}x+b_j)}
\end{equation}
which generalizes the function $W$ from \cite[section~3]{KPCMFT2016}. Here $A_i$, $B_j$, $\alpha_i$, $\beta_j$ and $\theta$ will always be assumed strictly positive, while $a_i$ and $b_j$ are non-negative.  Letting $q\uparrow1$ in \eqref{eq:Wq}
we get the $\theta=1$ case of $V(x)$.   A particular case of $V(x)$ has been recently demonstrated to be l.c.m. in \cite[Proof of Theorem~1.5]{Zhu} as a part of an investigation of Stieltjes moment sequences.

The following result is a straightforward generalization of \cite[Lemma~1]{KPCMFT2016}.  The proof repeats the proof of \cite[Lemma~1]{KPCMFT2016} \emph{mutatis mutandis} an will be omitted.
\begin{lem}\label{lm:V2prime-cm}
The function $(\log V)''$ is completely monotonic if and only if
\begin{equation}\label{eq:Q(u)}
Q(u):=\sum_{i=1}^p\frac{\alpha_{i}e^{-a_iu/A_i}}{1-e^{-u/A_i}}-\sum_{j=1}^s\frac{\beta_{j}e^{-b_ju/B_j}}{1-e^{-u/B_j}}\ge0
~\text{for all}~u>0.
\end{equation}
In the affirmative case
\begin{equation}\label{eq:logW2prime-int}
(\log V(x))''=\int_0^\infty e^{-ux}uQ(u)du.
\end{equation}
\end{lem}

Theorem~4 from \cite{KPCMFT2016} admits the following generalization. A similar proof is also omitted.
\begin{thm}\label{th:Vlcm}
Let $x>0$ and  $A_i, B_j, \alpha_i, \beta_j$  be strictly positive and $a_i$ and $b_j$ be non-negative.
The function $V(x)$ is logarithmically  completely monotonic if and only if condition \eqref{eq:Q(u)} holds true and
\begin{equation}\label{eq:balance-enthropy}
\sum_{i=1}^p\alpha_iA_i=\sum_{j=1}^{s}\beta_jB_j,\quad \rho:=\prod_{i=1}^{p}A_{i}^{\alpha_iA_i}
\prod_{j=1}^{s}B_{j}^{-\beta_jB_j}\leq\theta.
\end{equation}
In the affirmative case
$$
(-\log V(x))'=\int_{0}^{\infty}e^{-xu}Q(u)du+\log(\theta/\rho).
$$
\end{thm}

We find the necessary conditions for logarithmic complete monotonicity of $V$ in the following corollary. The proof goes along the same lines as the proof of \cite[Corollary~1]{KPCMFT2016} and will be omitted.
\begin{cor}\label{cr:necessary}
In addition to \eqref{eq:balance-enthropy} the following conditions are necessary for $V$ to be logarithmically completely monotonic\emph{:}
\begin{itemize}
\item[ \emph{(a)}] $\sum_{j=1}^s\beta_j(b_j-\frac{1}{2})-\sum_{i=1}^p\alpha_i(a_i-\frac{1}{2})\geq0$\emph{;}
\item[ \emph{(b)}] $\min_{1\le{i}\le{p}}(a_i/A_i)\leq\min_{1\le{j}\le{s}}(b_j/B_j)$; in case of equality it is further necessary that $\sum_{k\in{I}}\alpha_k\ge\sum_{m\in{J}}\beta_{m}$, where $I$, $J$ are the sets of indices for which minima are attained on the left and on the right, respectively.
\end{itemize}	
\end{cor}
The following theorem providing some sufficient conditions is a generalization of \cite[Theorem~5]{KPCMFT2016} and has a similar proof.
\begin{thm}\label{th:sufficient-old}
Inequality \eqref{eq:Q(u)} is true if any of the following sets of conditions holds\emph{:}
\begin{itemize}
\item[\emph{(a)}] 	$\sum_{i=1}^p\alpha_iA_i=\sum_{j=1}^s\beta_jB_j$ and $\max_{1\leq i\leq p}(a_i/A_i)\leq \min_{1\leq j\leq s}(b_j-1)/B_j$;

\item[\emph{(b)}]  $p=s$, $\sum_{i=1}^p\alpha_iA_i=\sum_{i=1}^p\beta_iB_i$ with $\alpha_{i}A_i\ge\beta_{i}B_i$ for $i=1,\ldots,p-1$, and $\max_{1\le{j}\le{p-1}} b_{j}/B_{j}\le(b_{p}-1)/B_{p}$, $a_i/A_i\le(b_i-1)/B_i$ for $i=1,\ldots,p$.
\end{itemize}
\end{thm}

Below we present two  new classes of completely monotonic gamma ratios not considered in \cite{KPCMFT2016}.
We will need a slight modification of a particular case of the easy direction in Sherman's theorem \cite[Theorem~1]{Borcea}, \cite[Theorem~4.7.3]{NP} given in the proposition below. Just like in the classical situation our proof is by application of  Jensen's inequality. See, for instance, \cite[Theorem~1]{Borcea}, \cite[Theorem~4.7.3]{NP}, \cite[Remark~1.4]{BradPecaric} for details and related results.

\begin{thm}\label{th:Sherman}
Suppose the real vectors $\mathbf{x}=(x_1,\ldots,x_{p})\in[\alpha,\beta]^{p}$, $\mathbf{y}=(y_1,\ldots,y_{s})\in[\alpha,\beta]^{s}$ and nonnegative vectors
$\mathbf{c}=(c_1,\ldots,c_{p})\in[0,\infty)^{p}$, $\mathbf{d}=(d_1,\ldots,d_{s})\in[0,\infty)^{s}$
satisfy the inequalities
\begin{equation}\label{eq:sherman-condition}
y_j\ge\sum\limits_{i=1}^{p}x_{i}h_{ji},~~~j=1,\ldots,s~~\text{and}~~c_{i}\ge\sum\limits_{j=1}^{s}d_{j}h_{ji},~~~i=1,\ldots,p
\end{equation}
for some nonnegative $s\times{p}$ row stochastic matrix $(h_{ji})$, i.e. such that $h_{ji}\ge0$ and $\sum_{i=1}^{p}h_{ji}=1$ for $j=1,\ldots,s$. Then for any convex decreasing function $f:[\alpha,\beta]\to[0,\infty)$ we have
\begin{equation}\label{eq:sherman}
\sum\limits_{j=1}^{s}d_{j}f(y_j)\le\sum\limits_{i=1}^{p}c_{i}f(x_{i}).
\end{equation}
\end{thm}
\begin{proof}
Indeed,
\begin{multline*}
\sum\limits_{j=1}^{s}d_{j}f(y_j)\le\sum\limits_{j=1}^{s}d_{j}f\left(\sum\limits_{i=1}^{p}x_{i}h_{ji}\right)
\le\sum\limits_{j=1}^{s}d_{j}\sum\limits_{i=1}^{p}h_{ji}f\left(x_{i}\right)
\\
=\sum\limits_{i=1}^{p}f\left(x_{i}\right)\sum\limits_{j=1}^{s}d_{j}h_{ji}
\le\sum\limits_{i=1}^{p}c_{i}f\left(x_{i}\right).
\end{multline*}
The first inequality is due to decrease of $f$ in view of the first condition in \eqref{eq:sherman-condition}; the second inequality is Jensen's inequality \cite[Lemma~1.1.11]{NP} valid since $f$ is convex and $(h_{ij})$ is row stochastic; finally, the ultimate inequality is true due to the second condition in \eqref{eq:sherman-condition} and nonnegativity of $f$.
\end{proof}

By restricting generality we can get rid of the matrix $(h_{ji})$ in the above theorem and formulate
the hypothesis directly in terms of parameters.

\begin{cor}\label{cr:sherman-simplified}
Let $\mathbf{x}$, $\mathbf{y}$, $\mathbf{c}$, $\mathbf{d}$ retain their meaning from Theorem~\emph{\ref{th:Sherman}}.
Suppose further that
\begin{equation}\label{eq:sherman-simplified}
\sum\limits_{i=1}^{p}c_{i}=\sum\limits_{j=1}^{s}d_{j}=D~~\text{and}~~ y_j\ge\frac{1}{D}\sum\limits_{i=1}^{p}c_{i}x_{i},~~~j=1,\ldots,s.
\end{equation}
Then inequality \eqref{eq:sherman} holds for any convex decreasing function $f:[\alpha,\beta]\to[0,\infty)$.
\end{cor}
\begin{proof}
Indeed, choosing $h_{ji}=c_i/D$ for all $i=1,\ldots,p$ and $j=1,\ldots,s$ we get a row stochastic matrix.
It remains to apply  Theorem~\ref{th:Sherman}.
\end{proof}

The following theorem furnishes a large number of examples of logarithmically completely monotonic gamma ratios of the form \eqref{eq:V-defined}.

\begin{thm}\label{th:sufficient}
Suppose $(h_{ji})$ is an arbitrary row stochastic matrix. Assume further that the positive numbers $(\alpha_1,\ldots\alpha_p)$, $(\beta_{1},\ldots,\beta_{s})$, $(A_1,\ldots,A_{p})$, $(B_1,\ldots,B_{s})$ and nonnegative numbers $(a_1,\ldots,a_p)$, $(b_1,\ldots,b_s)$ satisfy the conditions
\begin{equation}\label{eq:sufficientSherman}
\begin{split}
&\alpha_{i}A_{i}\ge\sum\limits_{j=1}^{s}\beta_{j}B_jh_{ji}~\text{for}~i=1,\ldots,p,
\\
&b_{j}\ge B_{j}\sum\limits_{i=1}^{p}\frac{a_{i}}{A_i}h_{ji}+1~\text{for}~j=1,\ldots,s.
\end{split}
\end{equation}
Then $Q(u)\ge0$ for all $u>0$, where $Q(u)$ is defined in \eqref{eq:Q(u)}.
\end{thm}
\begin{proof}
By the mean value theorem we have
$$
\sum_{i=1}^{p}\frac{\alpha_{i}e^{-a_iu/A_i}}{1-e^{-u/A_i}}=\sum_{i=1}^{p}\frac{\alpha_{i}}{e^{a_iu/A_i}-e^{(a_i-1)u/A_i}}
=\sum_{i=1}^{p}\frac{\alpha_{i}}{(u/A_i)e^{u\xi_i}}=\frac{1}{u}\sum_{i=1}^p\alpha_{i}A_ie^{-u\xi_i}
$$
where $\xi_{i}\in((a_i-1)/A_i, a_i/A_i)$ and similarly
$$
\sum_{j=1}^{s}\beta_{j}\frac{e^{-b_{j}u/B_{j}}}{1-e^{-u/B_j}}=\frac{1}{u}\sum_{j=1}^{s}\beta_{j}B_{j}e^{-u\eta_j},
$$	
where $\eta_{j}\in((b_{j}-1)/B_{j}, b_{j}/B_{j})$. Hence, for $Q(u)$ we get
\begin{equation}\label{eq:Qmeanvalue}
uQ(u)=\sum_{i=1}^{p}\alpha_{i}A_{i}e^{-u\xi_i}-\sum_{j=1}^{s}\beta_{j}B_{j}e^{-u\eta_j}
\ge\sum_{i=1}^{p}\alpha_{i}A_{i}e^{-ua_{i}/A_{i}}-\sum_{j=1}^{s}\beta_{j}B_{j}e^{-u(b_{j}-1)/B_{j}}.
\end{equation}	
We are now in the position to apply Theorem~\ref{th:Sherman} with $d_{j}=\beta_{j}B_{j}$, $y_{j}=(b_{j}-1)/B_{j}$, $j=1,\ldots,s$, $c_{i}=\alpha_{i}A_{i}$,  $x_{i}=a_{i}/A_{i}$, $i=1,\ldots,p$, and $f(x)=e^{-ux}$. For any fixed $u>0$ this function is decreasing and convex.
\end{proof}

\begin{cor}\label{cr:Vlcmnew}
Suppose $V(x)$ is defined in \eqref{eq:V-defined}. If conditions \eqref{eq:sufficientSherman} are satisfied,
then $(\log{V})''$ is completely monotonic. If, moreover, conditions \eqref{eq:balance-enthropy} are satisfied,
then $V(x)$ is l.c.m.
\end{cor}

An application of Corollary~\ref{cr:sherman-simplified} leads immediately to
\begin{cor}\label{cr:Vlcm-simplified}
Suppose conditions \eqref{eq:balance-enthropy} are satisfied and
$$
(b_{j}-1)\sum\limits_{k=1}^{s}\beta_{k}B_{k}\ge B_{j}\sum\limits_{i=1}^{p}\alpha_{i}a_{i}~\text{for}~j=1,\ldots,s.
$$
Then $V(x)$ defined in \eqref{eq:V-defined} is l.c.m.
\end{cor}

Next, we present another class of logarithmically completely monotonic functions of the form \eqref{eq:V-defined}.
Assume $p=s$ and $a_j=b_j=a$, $A_j=1/\alpha_j$, $B_j=1/\beta_j$ for all $j=1,\ldots,p$. Then
$$
\widehat{V}(x)=\prod_{j=1}^{p}\frac{\Gamma^{\alpha_{j}}(x/\alpha_{j}+a)\alpha_{j}^x}{\Gamma^{\beta_{j}}(x/\beta_{j}+a)\beta_{j}^x}
=\widehat{W}(x)\prod_{j=1}^{p}(\alpha_{j}/\beta_{j})^x,
$$
so that
\begin{multline*}
(\log\widehat{V})'=\sum\limits_{j=1}^{p}[\psi(x/\alpha_{j}+a)-\psi(x/\beta_{j}+a)+\log(\alpha_{j})-\log(\beta_{j})]
\\
=(\log\widehat{W})'+\sum\limits_{j=1}^{p}[\log(\alpha_{j})-\log(\beta_{j})]
\end{multline*}
and
$$
(\log\widehat{V})''=(\log\widehat{W})''.
$$
We have the following proposition.
\begin{thm}\label{th:Vhatlcm}
Suppose $a\ge1$ and the conditions
\begin{equation}\label{eq:A-1B-1maj}
\begin{split}
&0<\alpha_{1}\leq\alpha_{2}\leq\ldots\leq\alpha_{p},~~~0<\beta_{1}\leq\beta_{2}\leq\ldots\leq\beta_{p},
\\[7pt]
&\sum\nolimits_{j=1}^{k}\alpha_{j}\leq\sum\nolimits_{j=1}^{k}\beta_{j}~~~~~\text{for}~k=1,\ldots,p,
\end{split}
\end{equation}
are satisfied.  Then $(\log\widehat{W})'$ is a Bernstein function and $\widehat{V}$ is logarithmically completely monotonic.
\end{thm}
\begin{proof} As $(\log\widehat{W})'(0)=0$ this function is Bernstein if and only if $(\log\widehat{W})''$ is completely monotonic. According to Lemma~\ref{lm:V2prime-cm} this will be the case if and only if
$$
\widehat{Q}(u):=\sum_{j=1}^p\left\{\frac{\alpha_{j}e^{-au\alpha_j}}{(1-e^{-u\alpha_j})}
-\frac{\beta_{j}e^{-au\beta_j}}{(1-e^{-u\beta_j})}\right\}\ge0
$$
on $[0,\infty)$. To prove this inequality we need the following lemma
\begin{lem}\label{lm:phi-convex}
Suppose $\delta\ge\gamma>0$.  Then  function
$$
\phi_{\delta,\gamma}(t)=\frac{te^{-\delta{t}}}{1-e^{-\gamma{t}}}
$$
is decreasing and convex on $[0,\infty)$.
\end{lem}
\begin{proof}
Note first that if $f(t)$ is positive, convex and decreasing, then so is $e^{-\lambda{t}}f(t)$ for  $\lambda>0$.
Indeed, decrease is obvious, and it remains  to show that
\begin{align*}
& e^{-\lambda{t_1}}f(t_1)+e^{-\lambda{t_2}}f(t_2)\ge 2e^{-\lambda(t_1+t_2)/2}f((t_1+t_2)/2)
\\
\Leftrightarrow~& e^{\lambda(t_2-t_1)/2}f(t_1)+e^{\lambda(t_1-t_2)/2}f(t_2)\ge 2f((t_1+t_2)/2)
\end{align*}
for $t_1<t_2$.  Writing $q=e^{\lambda(t_2-t_1)/2}>1$, in view of $f(t_2)<f(t_1)$ and convexity of $f$, we will have
\begin{multline*}
qf(t_1)+q^{-1}f(t_2)=f(t_1)+f(t_2)+(q-1)(f(t_1)-q^{-1}f(t_2))
\\
>f(t_1)+f(t_2)+(q-1)(f(t_1)-f(t_2))>f(t_1)+f(t_2)\ge 2f((t_1+t_2)/2).
\end{multline*}
As $\phi_{\delta,\gamma}(t)=e^{-(\delta-\gamma)t}\phi_{\gamma,\gamma}(t)$
it is then sufficient to establish the lemma for $\phi_{\gamma,\gamma}(t)$.  In turn,
$\phi_{\gamma,\gamma}(t)=(1/\gamma)\phi_{1,1}(\gamma{t})$, so that it suffices to establish the claim for
$\phi_{1,1}(u)$.  The function $\phi_{1,1}(u)$ is the exponential generating function of Bernoulli numbers and is well-known to be decreasing and convex on the whole real line.
\end{proof}
We now return to the proof of Theorem~\ref{th:Vhatlcm}. Inequality $\widehat{Q}(u)\ge0$ in terms of $\phi_{\alpha,\beta}(t)$ takes the form
$$
\sum_{j=1}^p\phi_{au,u}(\beta_j)\le\sum_{j=1}^p\phi_{au,u}(\alpha_j).
$$
According to \cite[Proposition~4.B.2]{MOA} the above lemma implies that this inequality
is true if $a\ge1$ and the conditions \eqref{eq:A-1B-1maj} are satisfied.

Finally, to prove that $\widehat{V}$ is logarithmically completely monotonic it remains to show that
$$
\lim\limits_{x\to\infty}(-\log\widehat{V})'=
\lim\limits_{x\to\infty}(-\log\widehat{W})'+\sum\limits_{j=1}^{p}[\log(\beta_{j})-\log(\alpha_{j})]\ge0.
$$
Using the asymptotic formula
$$
\psi(x)\sim \log(x)-\frac{1}{2x}+O(x^{-2})~~\text{as}~~x\to\infty,
$$
we can write
$$
\psi(Cx+c)=\log(x)+\log(C)+\bigg(c-\frac{1}{2}\bigg)\frac{1}{Cx}+O(x^{-2})\quad \text{as}\quad x\to\infty.
$$
Then, as $x\to\infty$
$$
(-\log\widehat{W}(x))'=\sum_{j=1}^p[\psi(x/\beta_{j}+a)-\psi(x/\alpha_{j}+a)]
=\sum_{j=1}^p[\log(\alpha_{j}-\log(\beta_{j})]+O(1/x)
$$
yielding
$$
\lim\limits_{x\to\infty}(-\log\widehat{V})'=0.
$$
\end{proof}

\textbf{Remark.}  The properties of the function $\phi_{\alpha,\beta}(u)$ were studied in great detail in \cite{QiGuoRGMIA}.
It appears, nevertheless, that the result of Lemma~\ref{lm:phi-convex} does not follow from investigations in \cite{QiGuoRGMIA}.

Alzer and Berg \cite{AlzBerg} (for $\delta=0$) and soon thereafter Leblanc and Johnson \cite{LeblancJohnson} (for $\delta\ge0$) studied the complete monotonicity properties of the function $\sum_{k=1}^{m}a_k\psi(b_kx+\delta)$ and its derivative.  In particular, an application of \cite[Lemma~2.1]{LeblancJohnson} to the function $(\log\widehat{V})'$ leads to the conclusion that $(\log\widehat{V})''$ is completely monotonic if $a\ge1/2$ and
$$
\max_{1\le{j}\le{p}}(\alpha_{j})\le\min_{1\le{j}\le{p}}(\beta_{j}).
$$
Our majorization conditions \eqref{eq:A-1B-1maj} are certainly much less restrictive than the above condition, but at the price of a slightly stronger assumption $a\ge1$ on the parameter $a$.  Application of \cite[Lemma~2.1]{LeblancJohnson} in its full generality leads to the following proposition.
\begin{thm}\label{th:LeblancJohnson}
Suppose conditions \eqref{eq:balance-enthropy} hold, $a\ge1/2$ and, moreover,
\\[5pt]
\emph{(a)} $A_1\ge{A_2}\ge\cdots\ge{A_p}\ge{B_1}\ge{B_2}\ge\cdots\ge{B_s}>0$,
\\[7pt]
\emph{(b)} $\alpha_{1}A_1\ge\alpha_{2}A_2\ge\cdots\ge\alpha_{p}A_p>0$,
\\[7pt]
\emph{(c)} $0<\beta_{1}B_1\le\beta_{2}B_2\le\cdots\le\beta_{s}B_s$.
\\[5pt]
Then, the function
$$
V(x)=\theta^{-x}\frac{\prod_{i=1}^p\Gamma^{\alpha_{i}}(A_{i}x+a)}{\prod_{j=1}^s\Gamma^{\beta_{j}}(B_{j}x+a)}
$$
is logarithmically completely monotonic.
\end{thm}

For $p=1$ we can be more precise.
\begin{thm}\label{th:Vp1}
The function
$$
\widehat{V}_1(x)=\frac{\Gamma^{\alpha}(x/\alpha+a)\alpha^{x}}{\Gamma^{\beta}(x/\beta+a)\beta^{x}}
$$
is l.c.m. and the function
$$
(\log\widehat{W}_1)'(x)=\psi(x/\alpha+a)-\psi(x/\beta+a)
$$
is a Bernstein function if and only if $\alpha\le\beta$ and, if $\alpha\ne\beta$, $a\ge1/2$.
\end{thm}
\begin{proof}
As conditions \eqref{eq:balance-enthropy} are satisfied for the function $\widehat{V}_1(x)$ and $\lim\limits_{x\to0}(\log\widehat{W}_1)'(x)=0$ the claim of the theorem is equivalent to the assertion that
$$
(\log\widehat{W})''(x)=(1/\alpha)\psi'(x/\alpha+a)-(1/\beta)\psi'(x/\beta+a)
$$
is completely monotonic which, by Lemma~\ref{lm:V2prime-cm}, is equivalent to the inequality
$$
Q_1(u)=\frac{{\alpha}e^{-au\alpha}}{1-e^{-u\alpha}}-\frac{{\beta}e^{-ua\beta}}{1-e^{-u\beta}}\ge0~\text{for all}~u>0.
$$
Hence, we need to prove that the above inequality is true if and only if $\alpha\le\beta$ and $a\ge1/2$ when $\alpha\ne\beta$.
Assume first that it is true. According to Corollary~\ref{cr:necessary}(b) this implies that  $\alpha\le\beta$.
Next, cross-multiplication yields the following form of the required inequality
$$
\sigma(u)={\alpha}e^{-au\alpha}-{\alpha}e^{-u(a\alpha+\beta)}-{\beta}e^{-ua\beta}+{\beta}e^{-u(a\beta+\alpha)}\ge0.
$$
An easy calculation shows that $\sigma(0)=\sigma'(0)=0$ and $\sigma''(0)=\alpha\beta(2a-1)(\beta-\alpha)$.
The case $\alpha=\beta$ is trivial, so we assume that $\alpha<\beta$. Then $2a-1\ge0$ is necessary for $\sigma(u)\ge0$
in the neighborhood of $u=0$.

In the opposite direction assume that $\alpha\le\beta$ and $a\ge1/2$.  Inequality $uQ_1(u)>0$ follows from the fact that the function $x\to xe^{-ax}/(1-e^{-x})$ is strictly decreasing on $(0,\infty)$ if $a\ge1/2$.
\end{proof}

\section{An application}

Motivated by some problems in logarithmic concavity of generic series containing ratios of rising factorials, we
consider monotonicity of certain ratios of products of gamma functions. We remark here that the decreasing gamma ratio
in the following theorem is not completely monotonic under conditions of the theorem.  As far as we know these monotonicity results are new.
\begin{thm}\label{th:monotonic-gammas}
Suppose conditions \eqref{eq:A-1B-1maj} are satisfied and denote $A_{i}=1/\alpha_{i}$, $B_{i}=1/\beta_{i}$, $i=1,\ldots,p$, and $Y>X>0$. Then
the function
\begin{equation}\label{eq:monotonic-gammas}
F(a)=\prod_{i=1}^{p}\frac{\Gamma(A_{i}X+a)\Gamma(B_{i}Y+a)}{\Gamma(A_{i}Y+a)\Gamma(B_{i}X+a)}
\end{equation}
is monotone decreasing on $[1,\infty)$.
\end{thm}
\begin{proof}
Indeed, by Theorem~\ref{th:Vhatlcm} the function
$$
(\log\widehat{W})'(x)=\sum\limits_{i=1}^{p}[\psi(A_{i}x+a)-\psi(B_{i}x+a)]
$$
is a Bernstein function  under the conditions of the theorem and, in particular, is increasing on $(0,\infty)$.
Hence,
$$
\frac{\partial}{\partial{a}}\log{F(a)}
=\sum\limits_{i=1}^{p}[\psi(A_{i}X+a)-\psi(B_{i}X+a)-(\psi(A_{i}Y+a)-\psi(B_{i}Y+a))]<0
$$
because $Y>X>0$.
\end{proof}
\textbf{Remark.} In view of Theorem~\ref{th:Vp1} for $p=1$ the the function $F(a)$ in \eqref{eq:monotonic-gammas} is decreasing on $[1/2,\infty)$ if $A\ge{B}$ and $Y>X>0$.

The $p=1$ case also leads to the following monotonicity result for a ratio of rising factorials.

\begin{cor}
Suppose $\delta>0$ and $n>m>0$ are integers.
Then the rational function
$$
F(a)=\frac{(a+\delta{m})_{m}}{(a+\delta{n})_{n}}
$$
is decreasing on $[1/2,\infty)$.
\end{cor}
\begin{proof}
Denote $X=m$, $Y=n$, $B=\delta$, $A=1+\delta$. Then as $(x)_k=\Gamma(x+k)/\Gamma(x)$, we obtain
$$
\frac{\partial}{\partial{a}}\log{F(a)}=\psi(AX+a)-\psi(BX+a)-(\psi(AY+a)-\psi(BY+a))]<0
$$
for all $a\ge1/2$ in view of the remark below Theorem~\ref{th:monotonic-gammas}.
\end{proof}

For $p=2$ Theorem~\ref{th:monotonic-gammas} asserts that the function
$$
a\to\frac{\Gamma(A_{1}X+a)\Gamma(B_{1}Y+a)\Gamma(A_{2}X+a)\Gamma(B_{2}Y+a)}{\Gamma(A_{1}Y+a)\Gamma(B_{1}X+a)\Gamma(A_{2}Y+a)\Gamma(B_{2}X+a)}
$$
is decreasing on $[1,\infty)$ if $Y>X>0$, $0<A_1^{-1}\le{A_2^{-1}}$, $0<B_1^{-1}\le{B_2^{-1}}$, $A_1^{-1}\le{B_1^{-1}}$ and
$A_1^{-1}+A_2^{-1}\le B_1^{-1}+B_2^{-1}$.  On denoting $\mu_1=A_1X$, $\mu_2=A_2X$, $\mu_3=A_1Y$, $\mu_4=A_2Y$, $\nu_1=B_1Y$, $\nu_2=B_2Y$, $\nu_3=B_1X$, $\nu_4=B_2X$, the function in question becomes
$$
a\to\frac{\Gamma(\mu_1+a)\Gamma(\mu_2+a)\Gamma(\nu_1+a)\Gamma(\nu_2+a)}{\Gamma(\mu_3+a)\Gamma(\mu_4+a)\Gamma(\nu_3+a)\Gamma(\nu_4+a)}.
$$
This function is decreasing on $[1,\infty)$ if the following conditions are satisfied:
\begin{equation*}
\begin{split}
&\text{(a)}~~ \mu_i,\nu_i>0,~~i=1,\ldots,4;
\\[6pt]
&\text{(b)}~~ \frac{\mu_3}{\mu_1}=\frac{\mu_4}{\mu_2}=\frac{\nu_1}{\nu_3}=\frac{\nu_2}{\nu_4}>1~\text{(since $Y>X$)};
\\[6pt]
&\text{(c)}~~ \mu_2\le\mu_1~\text{and}~\nu_2\le\nu_1~\left(\text{since}~A_2\le{A_1}~\text{and}~B_2\le{B_1}\right);
\\[6pt]
&\text{(d)}~~\nu_1\le\mu_3~\text{and}~\frac{1}{\mu_1}+\frac{1}{\mu_2}\le\frac{1}{\nu_3}+\frac{1}{\nu_4}
\left(\text{since}~B_1\le{A_1}~\text{and}~A_1^{-1}+A_2^{-1}\le B_1^{-1}+B_2^{-1}\right).
\end{split}
\end{equation*}
Note that in view of (b) conditions (c) also imply that $\mu_{4}\le{\mu_3}$ and $\nu_{4}\le\nu_{3}$, while conditions (d) imply that $\nu_{3}\le\mu_{1}$ and $\mu_{3}^{-1}+\mu_{4}^{-1}\le\nu_{1}^{-1}+\nu_{2}^{-1}$.
We further remark that we can recover the parameters $A_1$, $A_2$, $B_1$, $B_2$, $X$, $Y$ from $\mu_i$, $\nu_j$ by setting $X=1$ and $Y=$the common value of the ratios in (b). The remaining parameters are then immediate from definitions of $\mu_i$, $\nu_j$.

\end{document}